\documentclass[12pt]{amsart}
\usepackage[all]{xy}
\usepackage{graphicx}
\usepackage{amsmath,amsxtra,amssymb,latexsym, amscd,amsthm,enumerate}
\usepackage{mathtools}
\usepackage{indentfirst}
\usepackage[mathscr]{eucal}
\usepackage{stackrel}
\usepackage{float} %%% fixing all Figures
\usepackage{color}
\usepackage{url}
\usepackage{microtype}

\usepackage{color}
\usepackage{hyperref}
\definecolor{darkblue}{RGB}{0,0,160}
\hypersetup{
	colorlinks,%
	citecolor=black,%
	filecolor=black,%
	linkcolor=darkblue,%
	urlcolor=darkblue
}

\newtheorem{thm}{Theorem}[section]
\newtheorem{cor}[thm]{Corollary}
\newtheorem{lem}[thm]{Lemma}
\newtheorem{prop}[thm]{Proposition}

\theoremstyle{definition}

\newtheorem{exam}[thm]{Example}

\DeclareMathOperator{\reg}{reg}
\DeclareMathOperator{\chara}{char}
\DeclareMathOperator{\pd}{pdim}

\DeclareMathOperator{\inn}{in_<}
\DeclareMathOperator{\depth}{depth}

\newcommand {\RR} {\mathbb R}

\newcommand\defas{\coloneqq}

\newcommand {\kk} {\Bbbk}
\def\m {\mathfrak m}
\def\p {\mathfrak p}

\def\I {\mathcal I}
\def\J {\mathcal J}

\usepackage{setspace}

\begin{document}
\title[Parity binomial edge ideals of complete graphs]
{Hilbert--Poincar\'{e} series of parity binomial edge ideals and
  permanental ideals of complete graphs}

\author{Do Trong Hoang}
\email{dthoang@math.ac.vn}
\address{Institute of Mathematics, Vietnam Academy of Science and Technology, 18 Hoang Quoc Viet, 10307 Hanoi, Vietnam}

\author{Thomas Kahle}
  \email{thomas.kahle@ovgu.de}
 \address{Fakult\"at f\"ur Mathematik, Otto-von-Guericke Universit\"at, Universit\"atsplatz  2,
D-39106 Magdeburg, Germany}
\thanks{}
\date{}
\dedicatory{}
\commby{}
% -----------------------------------------------------------
\begin{abstract} We give an explicit formula for the
  Hilbert--Poincar\'{e} series of the parity binomial edge ideal of a
  complete graph~$K_{n}$ or equivalently for the ideal generated by
  all $2\times 2$-permanents of a $2\times n$-matrix.  It follows that
  the depth and Castelnuovo--Mumford regularity of these ideals are
  independent of~$n$.
\end{abstract}
% -----------------------------------------------------------
\maketitle

\section{Introduction}

Let $R = \kk[x_{1},\dots,x_{n},y_{1},\dots,y_{n}]$ be a standard
graded polynomial ring in $2n$ indeterminates.  The \emph{parity
  binomial edge ideal} of an undirected simple graph $G$ on
$[n]=\{1,\dots,n\}$ is
\[
  \I_G= \left( x_ix_j-y_iy_j \mid \{i,j\}\in E(G) \right) \subset R,
\]
where $E(G)$ is the edge set of $G$.  This ideal was defined and
studied in \cite{KSW} in formal similarity to the binomial edge ideals
of~\cite{HHHKR} and~\cite{O}.  If $\chara(\kk)\neq 2$, then the linear
coordinate change $x_{i} \mapsto (x_{i}-y_{i})$ and
$y_{i} \mapsto x_{i}+y_{i}$ turns this ideal into the
\emph{permanental edge ideal}
\[
  \left( x_{i}y_{i} + x_{j}y_{j} \mid \{i,j\}\in E(G) \right) \subset R.
\]

We aim to understand homological properties of these ideals and we
view such understanding as helpful in the context of complexity theory
and the dichotomy of permanents and determinants.  In linear algebra
it is known that determinants can be evaluated quickly with Gaussian
elimination, but permanents are $\#$P-complete and thus NP-hard to
evaluate.  This complexity distinction is also visible for ideals
generated by determinants and permanents, as the permanental versions
are often much harder to analyze and have nice properties much more
rarely.  For details and history we recommend~\cite{LS} which treats
ideals of $2\times 2$-permanents of $m\times n$-matrices case in
detail.

$2\times 2$-permanental ideals also arise from the study of orthogonal
embeddings of graphs in~$\RR^{2}$ as the Lov\'asz--Saks--Schrijver
ideals of~\cite{HMMW}.  That paper also contains information about
radicality and Gr\"obner bases of parity binomial edge ideals.  Badiane,
Burke and Sk\"oldberg proved in \cite{BBS} that the universal
Gr\"obner basis and the Graver basis coincide for parity binomial edge
ideals of complete graphs.  The case of bipartite graphs is also
special, as then binomial edge ideals and parity binomial edge ideals
agree up to a linear coordinate change.  A coherent presentation of
our knowledge about these binomial ideals can be found in \cite{HHO},
in particular Chapter~7.

In this paper we are concerned with permanental ideals of
$2\times n$-matrices, but switch to the representation as parity
binomial edge ideals of complete graphs, as this seems easier to
analyze.  For example, the permanental ideal contains monomials by
\cite[Lemma~2.1]{LS} and these make the combinatorics more
opaque~\cite{KM}.  Due to the linear coordinate change, our
computations of homological invariants are valid for both ideals
unless $\chara(\kk) = 2$, in which case the permanental ideal and the
determinantal ideal agree.

The binomial edge ideal of a complete graph, also known as the
standard determinantal ideal of a generic $2\times n$-matrix, is well
understood.  It has a linear minimal free resolution independent
of~$n$, constructed explicitly by Eagon and Northcott~\cite{EN}.
Parity binomial edge ideals of complete graphs do not have a linear
resolution and their Betti numbers have no obvious explanation.

\setcounter{MaxMatrixCols}{20}
\begin{exam}\label{ex:betti}
  The package \textsc{BinomialEdgeIdeals} in Macaulay2~\cite{M2}
  easily generates the following Betti table of $\I_{K_{7}}$.  The
  Betti table agrees with the Betti table of a permanental ideal of a
  generic $2\times 7$-matrix.
  \[
    \begin{matrix}
                    & 0        & 1        & 2        & 3        & 4        & 5        & 6        & 7        & 8        & 9        & 10       & 11       \\
      \text{total:} & 1        & 21       & 455      & 1925     & 4256     & 6111     & 6160     & 4466     & 2289     & 784      & 161      & 15       \\
      \text{0:}     & 1        & \text{.} & \text{.} & \text{.} & \text{.} & \text{.} & \text{.} & \text{.} & \text{.} & \text{.} & \text{.} & \text{.} \\
      \text{1:}     & \text{.} & 21       & \text{.} & \text{.} & \text{.} & \text{.} & \text{.} & \text{.} & \text{.} & \text{.} & \text{.} & \text{.} \\
      \text{2:}     & \text{.} & \text{.} & 455      & 1890     & 3976     & 5166     & 4410     & 2520     & 945      & 210      & 21       & \text{.} \\
      \text{3:}     & \text{.} & \text{.} & \text{.} & 35       & 280      & 945      & 1750     & 1946     & 1344     & 574      & 140      & 15       \\
    \end{matrix}
  \]
\end{exam}

From computations for the first few $n$ one can observe that the
Castelnuovo--Mumford regularity (the index of the last row of the
Betti table) of $R/\I_{K_{n}}$ appears to be independent of $n \ge 4$
too, but now $\reg R/\I_{K_{n}} = 3$ (see Section~\ref{s:basics} for
definitions).  This was conjecture by the second author and
Kr\"usemann~\cite[Remark~2.15]{KK} and is now our
Theorem~\ref{thm:poincare}.  Our main results are explicit formulas
for the Hilbert--Poincar\'{e} series, the depth, the
Castelnuovo--Mumford regularity, and some extremal Betti numbers in
the case of a complete graph.  The proof of our theorem relies on good
knowledge of the primary decomposition of $\I_{K_{n}}$ from~\cite{KSW}
and the resulting exact sequences.  At the moment it is not clear if
the techniques can be generalized to other graphs or maybe even yield
the conjectured upper bound $\reg (R/\I_{G}) \le n$ from
\cite[Remark~2.15]{KK}.

\section{Basics of (parity) binomial edge ideals}
\label{s:basics}
Throughout this paper, let $G$ be a simple (i.e.\ finite, undirected,
loopless and without multiple edges) graph on the vertex set
$V(G)=[n]\defas \{1,\ldots,n\}$.  Let $E(G)$ denote the set of edges
of~$G$.  Each graded $R$-module and in particular $R/\I_{G}$ has a
minimal graded free resolution
\[
  0\leftarrow R/\I_G \leftarrow  \bigoplus_{j}
  R(-j)^{\beta_{0,j}(R/\I_G)} \leftarrow \cdots \leftarrow
  \bigoplus_{j} R(-j)^{\beta_{p,j}(R/\I_G)} \leftarrow 0.
\]
where $R(-j)$ denotes the free $R$-module obtained by shifting the
degrees of $R$ by~$j$.  The number $\beta_{i,j}(R/\I_G)$ is the
\emph{$(i,j)$-th graded Betti number} of~$R/\I_G$.  Let $H_{R/\I_G}$
be the Hilbert function of~$R/\I_G$.  The \emph{Hilbert--Poincar\'{e}
  series} of the $R$-module $R/\I_G$ is
\[
  HP_{R/\I_G}(t)= \sum_{i\ge 0} H_{R/\I_G}(i)t^i.
\]
By \cite[Theorem~16.2]{P}, this series has a rational expression
\[
  HP_{R/\I_G}(t)  = \frac{P_{R/\I_G}(t)}{(1-t)^{2n}}.
\]
The numerator
$P_{R/\I_G}(t) \defas \sum_{i=0}^p\sum_{j=0}^{p+r}
(-1)^i\beta_{i,j}(R/\I_G)t^j$ is the \emph{Hilbert--Poincar\'{e}
  polynomial} of~$R/\I_G$.  It encodes different homological
invariants of~$R/\I_{G}$ of which we are particulary interested in the
\emph{Castelnuovo--Mumford regularity}
\[
  \reg(R/\I_G) = \max\{j-i\mid \beta_{i,j}(R/\I_G)\ne 0\}
\]
and the \emph{projective dimension} of $R/\I_G$:
\[
  \pd(R/\I_G) = \max\{i\mid \beta_{i,j}(R/\I_G)\ne 0\text{ for some }
  j\}.
\]
In terms of Betti tables, the regularity is the index of the last
non-vanishing row, while the projective dimension is the index of the
last non-vanishing column of the Betti table.  Both are finite for any
$R$-module as $R$ is a regular ring.

The Auslander--Buchsbaum formula \cite[Theorem 2.15]{HHO} relates
depth and projective dimension over $R$ as
$\depth(R/\I_G) = 2n-\pd(R/\I_G)$.

The Castelnuovo--Mumford regularity and depth could also be computed
from vanishing of local cohomology.  Using that definition allows to
easily deduce some basic properties of the regularity and depth. For
instance, the regularity and depth behave well in a short exact
sequence.  The following lemma appears as~\cite[Corollary 18.7]{P}.
 
\begin{lem}\label{lem:Regseq} If
  $0 \to A \to B \to C\to 0$ is a short exact sequence of finitely
  generated graded $R$-modules with homomorphisms of degree~$0$, then
  \[
    P_{B}(t) =  P_{A}(t) + P_C(t), \text{ and }
  \]
  \begin{enumerate}
  \item $\reg(B)\le \max\{\reg(A), \reg(C)\}$, 
  \item $\reg(A)\le \max\{\reg(B), \reg(C)+1\}$,
  \item $\reg(C)\le \max\{\reg(A)-1, \reg(B)\}$,
  \item  $\depth(B)\ge \min\{\depth(A),\depth(C)\}$,
  \item  $\depth(A)\ge \min\{\depth(B),\depth(C)+1\}$,
  \item  $\depth(C)\ge \min\{\depth(A)-1,\depth(B)\}$.  
  \end{enumerate}
\end{lem}

\medskip

As with any binomial ideal, the saturation at the coordinate
hyperplanes plays a central role.  To this end, let
$g = \prod_{i\in [n]} x_iy_i$ and let
\[
  \J_G \defas \I_{G} : g^{\infty} \defas \bigcup_{t\ge 1} \I_G: g^t.
\]
By \cite[Proposition~2.7]{KSW}, the generators of the
saturation~$\J_G$ can be explained using walks in~$G$.  For our
purposes it suffices to know the following generating set which can be
derived from \cite[Section~2]{KSW}.

\begin{prop} \label{prop_Nbiparitite} If $G$ is a non-bipartite
  connected graph, then
  \[
    \J_{G} = (x_i^2-y_i^2\mid 1\le i\le n) + (x_iy_j-x_jy_i,\,  x_ix_j-y_iy_j\mid 1\le i<j\le n).
  \]
\end{prop}
 
\section{Parity binomial edge ideals of complete graphs}
We now consider the parity binomial edge ideal $\I_{K_{n}}$ of a
complete graph $K_{n}$ on $n\ge 3$ vertices.  For $1\le i<j\le n$, let
\[
  f_{ij}\defas x_iy_j-x_jy_i \quad \text{ and } \quad g_{ij} \defas
  x_ix_j - y_jy_i.
\]
The parity binomial edge ideal of the complete graph is
$\I_{K_n} = \left (g_{ij}\mid 1\le i<j\le n\right)$.

We need some further notation.  For any $I\subseteq [n]$ we denote
$\m_I \defas (x_i,y_i\mid i\in I)$.  Let
$\mathfrak{p}^{+} \defas (x_i+y_i\mid i\in [n])$ and
$\mathfrak{p}^{-} \defas (x_i-y_i\mid i\in [n])$.  Denote
$P_{ij} \defas (g_{ij})+ \m_{[n]\backslash \{i,j\}}$.  By
\cite[Theorem 5.9]{KSW}, there is a decomposition of $\I_{K_n}$ as
follows.
  
\begin{prop} \label{prop:primaryd} For $n\ge 3$, we have
  \[
    \I_{K_n} = \J_{K_n} \cap \bigcap_{1\le i<j\le n} P_{ij}.
  \]
  In particular, $\dim(R/\I_{K_n}) = n$.  
\end{prop}
We analyze $\I_{K_{n}}$ by regular sequences arising from successively
adding the polynomials~$f_{kn}$ or saturating with respect to them.
Let $I_0\defas\I_{K_n}$ and, inductively for $1\le k\le n-1$,
$I_k \defas I_{k-1} + (f_{kn})$.
\begin{lem} \label{subset} For $1\le k\le n-1$, we have
  \[
    I_{k-1} \subseteq  \bigcap_{1\le i<j\le n-1} P_{ij} \cap  \J_{K_n}
    \cap \bigcap_{t=k}^{n-1}P_{tn}.
  \]
\end{lem} 
\begin{proof} By Proposition~\ref{prop_Nbiparitite},
  $f_{1n}, \ldots, f_{(k-1)n}\in \J_{K_n}$.  Moreover, for all
  $(\ell,n)\ne (i,j)$ we have $f_{\ell n}\in P_{ij}$.  Thus
  \[
    (f_{1n}, \ldots, f_{(k-1)n}) \subseteq \bigcap_{1\le i<j\le n-1}
    P_{ij} \cap\J_{K_n} \cap \bigcap_{t=k}^{n-1}P_{tn}.
  \]
  Together with Proposition~\ref{prop:primaryd} the lemma is proven.
\end{proof} 
   
\begin{lem} \label{lem:01} For $1\le k\le n-1$, we have
  \[
    I_{k-1}: f_{kn} = P_{kn}.
  \]
  In particular, $\depth(R/(I_{k-1}: f_{kn}))= 3$,
  $\reg(R/(I_{k-1}: f_{kn})) =1$ and
  $P_{R/(I_{k-1}: f_{kn})}(t) = (1-t)^{2n-3}(1+t)$.
\end{lem} 
\begin{proof} One can check that $I_{k-1}: f_{kn} \supseteq P_{kn}$
  (in fact $\I_{K_{n}} : f_{kn} \supseteq P_{kn}$) by simple
  calculations like $x_{1}f_{kn} \equiv -y_{k}g_{1n} \mod \I_{K_{n}}$.
  Now, for all $(k,n)\ne (i,j)$, one can see that $f_{kn}$ is
  contained in both $P_{ij}$ and~$\J_{K_n}$.  By \cite[Lemma~4.4]{AM},
  $P_{ij}: f_{kn} = \J_{K_n}: f_{kn} = R$ and
  $P_{kn}: f_{kn} = P_{kn}$ because $P_{kn}$ is a prime that does not
  contain~$f_{kn}$.  Hence by Lemma~\ref{subset}, we have
  $I_{k-1}: f_{kn} \subseteq P_{kn}$ and thus
  $I_{k-1}: f_{kn} = P_{kn}$.

  Using this result, the invariants can be computed for the prime
  $P_{kn}$ as follows:
  $\depth(R/(I_{k-1}:f_{kn})) = \depth(R/P_{kn}) = 3$,
  $\reg(R/(I_{k-1}: f_{kn})) = \reg(R/P_{kn}) =1$, and
  $P_{R/(I_{k-1}: f_{kn})}(t) = P_{R/P_{kn}}(t) = (1-t)^{2n-3}(1+t)$.
\end{proof}

\begin{lem} \label{lem:02}
  \[
    I_{n-2}: (x_n +y_n) = \p^{-}\cap P_{n-1,n}.
  \]
  In particular, $\depth(R/(I_{n-2}: (x_n+y_n)))\ge 3$,
  $\reg(R/(I_{n-2}: (x_{n}+y_n)))\le 1$ and
  $P_{R/(I_{n-2}: (x_{n}+y_n))}(t) = (1-t)^n +2t(1-t)^{2n-3}$.
\end{lem}
\begin{proof} 
  For the lexicographic ordering on
  $\kk[x_1,\ldots, x_n, y_1,\ldots, y_n,t]$ induced by
  $x_1>\ldots>x_n>y_1>\ldots>y_n>t$, the Gr\"obner basis for
  $J=t\p^{-} + (1-t)P_{n-1,n}$ is
\begin{multline*}
  \mathcal{G} = \{(x_{n-1}-y_{n-1})t,\, (x_{n}-y_{n})t),\, x_{n-1}x_n-y_{n-1}y_n,\\
  x_i-y_i,\, (x_{n-1}-y_{n-1})y_i,\, (x_{n}-y_{n})y_i,\, (t-1)y_{i}
  \mid 1\le i\le n-2 \}.
\end{multline*}
Thus,
\begin{multline*}
  \p^{-}\cap P_{n-1,n} = \\ \left(x_{n-1}x_n-y_{n-1}y_n,\, x_i-y_i,\,
    (x_{n-1}-y_{n-1})y_i,\, (x_{n}-y_{n})y_i \mid 1\le i\le n-2\}\right).
\end{multline*}
This implies the containment
$ \p^{-}\cap P_{n-1,n} \subseteq I_{n-2}: (x_{n}+y_n)$.
Conversely, by Lemma~\ref{subset}
\[
  I_{n-2} \subseteq \bigcap_{1\le i<j\le n-1} P_{ij} \cap \J_{K_n}
  \cap  P_{n-1,n}.
\]
For all $1\le i<j\le n-1$, it is clear that $x_n+y_n\in P_{ij}$ and so
$P_{ij}:(x_n+y_n)=R$. By \cite[Lemma~4.4]{AM},
$P_{n-1,n}: (x_n+y_n)=P_{n-1,n}$.  Moreover, by
Proposition~\ref{prop_Nbiparitite}, we obtain that
$\J_{K_n}: (x_n+y_n) = \p^{-}$.  This implies that
$I_{n-2}: (x_n+y_n) \subseteq \p^{-} \cap P_{n-1,n}$ and thus the
conclusion $I_{n-2}: (x_{n}+y_n) = \p^{-}\cap P_{n-1,n}$.

In order to prove the second part, note that
\[
  \p^{-}+P_{n-1,n} = (x_{n-1}+y_{n-1}, x_n+y_n) + \m_{[n-2]}.
\]
Therefore one reads off $\depth(R/(\p^{-}+P_{n-1,n})) = 2$ and
$\reg(R/(\p^{-}+P_{n-1,n}))=0$.  It is clear that
$\depth(R/\p^{-}) =n$ and $\reg(R/\p^{-})=0$.  From the exact sequence
\[
  0\longrightarrow R/ (\p^{-}\cap P_{n-1,n})  \longrightarrow R/
  \p^{-}\oplus R/P_{n-1,n} \longrightarrow
  R/(\p^{-}+P_{n-1,n})\longrightarrow 0,
\]
we obtain, using Lemma~\ref{lem:Regseq}, that
\begin{align*}
  \depth(R/I_{n-2}: (x_{n}+y_n)) &=  \depth(R/ (\p^{-}\cap P_{n-1,n})) \ge   \min\{n,3,2+1\}=3,\\
  \reg(R/I_{n-2}: (x_{n}+y_n))  &= \reg(R/ (\p^{-}\cap P_{n-1,n})) \le  \max\{0,1,0+1\} =1,  
\end{align*}
and furthermore, 
\begin{align*} 
  P_{R/I_{n-2}: (x_{n}+y_n)}(t) &= P_{R/ \p^{-}}(t) + P_{R/P_{n-1,n}}(t) - P_{R/(\p^{-}+P_{n-1,n})}(t) \\
                                &= (1-t)^n +(1-t)^{2n-3}(1+t) - (1-t)^{2n-2}\\
                                &= (1-t)^n + 2t(1-t)^{2n-3}. \qedhere
\end{align*}
\end{proof}
  
\begin{lem}\label{lem:A} Let $J \defas (x_n+y_n, I_{n-2})$. Then
\begin{align*}
\depth(R/J) &\ge \min\{n,\depth(S/\I_{K_{n-1}})\},\\
\reg(R/J)&\le \max\{1, \reg(S/\I_{K_{n-1}})\}, 
\end{align*}
 and $P_{R/J}(t)   =   t(1-t)^n +   (1-t)^2 P_{S/\I_{K_{n-1}}}(t),$ where $S = \kk[x_i,y_i\mid 1\le i \le n-1]$. 
\end{lem}    
\begin{proof}
In order to prove the lemma, we first  check  two following claims:

\medskip

\noindent {\it Claim 1:}   $(J,x_n) = (x_n,y_n,  \I_{K_{n-1}}).$ 

Since $y_n = (x_n+y_n) - x_n\in (x_n,J)$ and
$\I_{K_{n-1}}\subseteq I_{n-2}$, we have
$(x_n,y_n, \I_{K_{n-1}}) \subseteq (J,x_n)$.  Conversely,
$x_n+y_n, g_{in}, f_{in} \in (x_n,y_n)$ for $1\le i\le n-1$ and thus
$(J,x_n)\subseteq (x_n,y_n, \I_{K_{n-1}})$.

\medskip
\noindent {\it Claim 2:}    $J: x_n  = \p^+.$

One can compute $x_n(x_i+y_i) = (x_ix_n-y_iy_n) + y_i(x_n+y_n)\in J$
for $1\le i\le n$, so that $x_n\p^{+} \subseteq J$ which implies that
$\p^{+} \subseteq J: x_n$.  Conversely, for $1\le i<j\le n$, we have
\begin{align*}
g_{ij} &= x_ix_j - y_iy_j = (x_i-y_i)x_j + y_i(x_j-y_j)= (x_i+y_i)x_j - y_i(x_j+y_j),\\
f_{ij} &= x_iy_j-x_jy_i = (x_i+y_i)y_j - y_i(x_j+y_j)= (x_i-y_i)y_j - y_i(x_j-y_j).
\end{align*}
Thus, by Proposition~\ref{prop_Nbiparitite},
$\J_{K_n} \subseteq \p^{+} \cap (x_1-y_1,\ldots,x_{n-1}-y_{n-1},
x_n,y_n)$ and
$f_{kn}\in \p^{+} \cap (x_1-y_1,\ldots,x_{n-1}-y_{n-1}, x_n,y_n)$ for
all $1\le k\le n-2$.  Together with Proposition~\ref{prop:primaryd},
\[
  J \subseteq \bigcap_{1\le i<j\le n-1} P_{ij} \cap \p^{+} \cap
  (x_1-y_1,\ldots,x_{n-1}-y_{n-1}, x_n,y_n).
\]
By \cite[Lemma~4.4]{AM},  $J: x_n \subseteq \p^+$ and thus the claim
holds.

\medskip Now, we turn to the proof of the lemma. By Claim~1,
\[
  \depth(R/(J,x_n)) = \depth(S/\I_{K_{n-1}}) \text{ and }
  \reg(R/(J,x_n)) = \reg(S/\I_{K_{n-1}}).\]
Moreover, by Claim~2, we have
\[
  \depth(R/J: x_n)=\depth(R/\p^+)=n \text{ and } \reg(R/J:x_n)
  =\reg(R/\p^+)=0.
\]
From the exact sequence
\[
  0 \longrightarrow R/(J: x_n)(-1)  \longrightarrow R/J
  \longrightarrow R/(J,x_n)\longrightarrow 0
\]
we obtain
\[ \depth(R/J) \ge \min\{n,\depth(S/\I_{K_{n-1}})\} \text{ and }
  \reg(R/J)\le \max\{1, \reg(S/\I_{K_{n-1}})\}.\] Moreover,
\begin{align*}
  P_{R/J}(t) &=  tP_{R/J: x_n}(t) +  P_{R/(J, x_n)}(t)  =   tP_{R/\p^+}(t) +  P_{R/(x_n,y_n,  \I_{K_{n-1}})}(t)\\
  &=  t(1-t)^{n} +    (1-t)^2 P_{S/\I_{K_{n-1}}}(t),
\end{align*} 
as required. 
\end{proof}
     
\begin{thm} \label{thm:poincare} The Hilbert--Poincar\'{e} polynomial
  of $R/\I_{K_{n}}$ is
 \[
   P_{R/\I_{K_n}}(t) =2(1-t)^n +  \Big[-1+3t + (\frac{n^2+n-6}{2})t^2
   + (\frac{n^2-3n+2}{2})t^3\Big](1-t)^{2n-3}.
 \]
 In particular, $\depth(R/\I_{K_n}) \ge 3$ and
 $\reg(R/\I_{K_n})\le 3$.
\end{thm} 
\begin{proof}
  The proof is by induction on~$n$. If $n=3$, then a simple
  calculation (e.g.\ in Macaulay2) gives the result.  Now assume
  $n\ge 4$.  For any $1\le k\le n-1$ there is an exact sequence
  \[
    0\longrightarrow R/(I_{k-1}: f_{kn})(-2) \xrightarrow{\cdot
      f_{kn}} R/I_{k-1} \longrightarrow R/I_{k} \longrightarrow 0.
  \]
  By Lemmas~\ref{lem:Regseq} and~\ref{lem:01},
  $ \depth(R/I_{k-1}) \ge \min\{3, \depth(R/I_{k})\}$,
  $\reg(R/I_{k-1}) \le \max\{3, \reg(R/I_{k})\}$ and
  $P_{R/I_{k-1}}(t) = t^2(1-t)^{2n-3}(1+t)+ P_{R/I_{k}}(t)$.  This
  implies that $\depth(R/I_0) \ge \min\{3,\depth(R/I_{n-2})\}$,
  $\reg(R/I_{0}) \le \max\{3,\reg(R/I_{n-2})\}$ and
  \[
    P_{R/I_{0}}(t) = (n-2)t^2(1-t)^{2n-3}(1+t)+ P_{R/I_{n-2}}(t).
  \]
  Now consider the following exact sequence
  \[
    0\longrightarrow R/(I_{n-2}: (x_{n}+y_n))(-1) \longrightarrow
    R/I_{n-2} \longrightarrow R/(x_n+y_n, I_{n-2}) \longrightarrow 0.
  \]
  Let $S\defas\kk[x_i,y_i\mid 1\le i\le n-1]$.  By Lemmas~\ref{lem:02}
  and~\ref{lem:A},
  $\depth(R/I_{n-2}) \ge \min\{3,\depth(S/\I_{K_{n-1}})\}$,
  $\reg(R/I_{n-2}) \le \max\{1, \reg(S/\I_{K_{n-1}})\}$ and
  \begin{align*}
    P_{R/I_{n-2}}(t) &=   tP_{R/I_{n-2}: x_n+y_n}(t) +  P_{R/(x_n+y_n,I_{n-2})}(t)\\
                     &=     2t(1-t)^n +    2t^2(1-t)^{2n-3}  +    (1-t)^2 P_{S/\I_{K_{n-1}}}(t).
  \end{align*}     
  The induction hypothesis yields $\depth(S/\I_{K_{n-1}})\ge 3$ and
  $\reg(S/\I_{K_{n-1}}) \le 3$.  Therefore $\depth(R/I_{n-2}) \ge 3$
  and $\reg(R/I_{n-2})\le 3$.  This is enough to conclude that
  $\depth(R/\I_{K_n}) \ge 3$ and $\reg(R/\I_{K_n})\le 3$.  Moreover,
  \begin{align*}
    P_{R/\I_{K_n}}(t) &= 2t(1-t)^n + \Big[(n-2)t^3 + nt^2\Big] (1-t)^{2n-3} +   (1-t)^2 P_{S/\I_{K_{n-1}}}(t). \\
                      &=  2t(1-t)^n + \Big[(n-2)t^3 + nt^2\Big] (1-t)^{2n-3}\\
                      & +    2(1-t)^{n+1} +  \Big[-1+3t + (\frac{n^2-n-6}{2})t^2 + (\frac{n^2-5n+6}{2})t^3\Big](1-t)^{2n-3}\\ 
                      &=2(1-t)^n +  \Big[-1+3t + (\frac{n^2+n-6}{2})t^2 + (\frac{n^2-3n+2}{2})t^3\Big](1-t)^{2n-3},
  \end{align*}
  as required.
\end{proof} 

If an ideal has a square-free initial ideal, its extremal Betti
numbers agree with that of the initial ideal by~\cite{CV}.  Although
the parity binomial edge ideal of complete graph cannot have a
square-free initial ideal (see \cite[Remark~3.12]{KSW}), the bottom
right Betti number agrees with that of the initial ideal for any term
order.

\begin{cor}
  \[
    \beta_{2n-3,2n}(R/\I_{K_n}) =\beta_{2n-3,2n}(R/\inn(\I_{K_n}))=
    \frac{n^2-3n+2}{2}.
  \]
  In particular,
  \[
    \reg(R/\I_{K_n}) = \reg(R/\inn(\I_{K_n})) = \depth(R/\I_{K_n}) =
    \depth(R/\inn(\I_{K_n})) =3.
  \]
\end{cor} 
\begin{proof}
  From Theorem~\ref{thm:poincare} we obtain
  $\beta_{p,p+r}(R/\I_{K_n}) =\frac{n^2-3n+2}{2}\ne 0$, where
  $p=\pd(R/\I_{K_n})$ and $r=\reg(R/\I_{K_n})$.  Thus, $p+r=2n$.
  Since $P_{R/\I_{K_n}}(t) = P_{R/\inn(\I_{K_n})}(t)$, we get
\[
  \reg(R/\I_{K_n}) = \reg(R/\inn(\I_{K_n})),\quad 
  \pd(R/\I_{K_n})=\pd(R/\inn(\I_{K_n}))  \text{ and }
\]
$\beta_{p,p+r}(R/\I_{K_n}) =\beta_{p,p+r}(R/\inn(\I_{K_n}))$.  On the
other hand, $r \le 3$ and $p\le 2n-3$ by the Auslander--Buchsbaum
formula.  Thus, $r = 3$ and $p = 2n-3$.
\end{proof}

\section*{Acknowledgement}

Do Trong Hoang was supported by the NAFOSTED Vietnam under grant
number 101.04-2018.307. This paper was done when he visited Department
of Mathematics, Otto-von-Guericke Universit\"at Magdeburg with the
support of Deut\-scher Aka\-de\-mischer Aus\-tausch\-dienst (DAAD).
Thomas Kahle acknowledges support from the DFG (314838170, GRK 2297
MathCoRe).

\end{document}